\documentclass[12pt]{article}
\usepackage{appendix}
\usepackage[a4paper]{geometry}
\usepackage[pdftex]{graphicx}
\usepackage{amssymb}
\usepackage[dvipsnames]{xcolor}
\definecolor{myblue}{RGB}{80,80,160}
\usepackage[utf8]{inputenc}
\usepackage[T1]{fontenc}
\usepackage[english]{babel}
\usepackage{array}
\usepackage{listings}
\usepackage{booktabs}
\usepackage{float}
\usepackage{amsmath}
\usepackage{cases}
\usepackage{mathtools}
\usepackage[overload]{empheq}
\makeatletter

\@addtoreset{equation}{section}
\makeatother
\usepackage{enumitem}
\usepackage{mathtools}
\usepackage{tikz}
\usetikzlibrary{shapes,patterns}
\usetikzlibrary{topaths,calc}
\usepackage{amsthm}
\newtheorem{thm}{Theorem}
\newtheorem{corollary}{Corollary}
\newtheorem{conjecture}{Conjecture}
\newtheorem{prop}{Proposition}
\newtheorem{lemme}{Lemma}
\newtheorem{obs}{Observation}
\usepackage[square,numbers]{natbib}
\usepackage{xurl}

\usepackage[mathscr]{euscript}

\title{\textbf{Completely Independent Spanning Trees in Split Graphs: Structural Properties and Complexity}}
\author{Mohammed Lalou, Nader Mbarek, Abdallah Skender, Olivier Togni\\
        Université Bourgogne Europe, LIB UR 7534, Dijon}
        
\begin{document}
\maketitle
\begin{abstract}
    We study completely independent spanning trees (CIST), \textit{i.e.}, trees that are both edge-disjoint and internally vertex-disjoint, in split graphs. We establish a correspondence between the existence of CIST in a split graph and some types of hypergraph colorings (panchromatic and bipanchromatic colorings) of its associated hypergraph, allowing us to obtain lower and upper bounds on the number of CIST. Using these relations, we prove that the problem of the existence of two CIST in a split graph is NP-complete. Finally, we formulate a conjecture on the bipanchromatic number of a hypergraph related to the results obtained for the number of CIST.
\end{abstract}

\section{Introduction}
The graphs considered in this paper are undirected, finite, and simple. For a graph $G$, the sets $V(G)$ and $E(G)$ are its vertex and edge sets, respectively. The set of vertices adjacent to a vertex $x$ is its {\em neighborhood} and is denoted $N_G(x)$, and the {\em degree} of $x$ in $G$ is $d_G(x) = |N_G(x)|$. The graph $G$ is {\em bipartite} if $V(G)$ admits a partition into two classes $V_1 \cup V_2$ such that each edge has one endpoint in each of the classes. Let $x,y$ be two vertices of $G$. A {\em $(x,y)$-path} is a sequence of vertices starting with $x$ and ending with $y$, where each consecutive pair is an edge of $E(G)$. A graph is {\em connected} if there is a path between any pair of its vertices, and it is {\em $k$-connected} if it remains connected after deleting any set of $k-1$ vertices. A graph is {\em acyclic} if it does not contain any cycles. A {\em tree} $T$ is an acyclic connected graph. A {\em spanning tree} $T$ of a graph $G$ is a tree with $V(T) = V(G)$ and $E(T) \subseteq E(G)$. A vertex $x \in V(T)$ is an {\em internal (or inner)} vertex if $d_T(x) \geq 2$, and otherwise, it is a {\em leaf}. If $x$ is an internal vertex of a tree $T$ and $y$ is a leaf of $T$, then $x$ is said to {\em cover} $y$ in $T$ when $x$ is adjacent to $y$ in $T$.

Let $G$ be a graph and $P_1,P_2$ two $(x,y)$-paths. $P_1$ and $P_2$ are {\em internally-disjoint} if they do not have an internal vertex in common; they are {\em openly-disjoint} if they are both edge-disjoint and internally-disjoint. Let $k\geq 2$ be an integer and $T_1, \dots , T_k$ be spanning trees of a graph $G$. If for any pair of vertices $x,y \in V(G)$, the $(x,y)$-paths in $T_1, \dots , T_k$ are openly-disjoint, then $T_1, \dots , T_k$ are called {\em completely independent spanning trees} (CIST for short) of $G$.

Independent spanning trees have several applications in various domains \cite{add1,add2} and serve as a key tool for enhancing the fault tolerance of communication systems. Indeed, a network topology is usually represented by a graph, where vertices represent devices and edges represent the physical connections between them. Providing $k$ CIST in such topologies ensures that the network remains operational even in the presence of $k-1$ vertex or edge failures. 

The practical interests of CIST necessitate characterizing them, identifying the maximum number of instances a graph may contain, and proposing constructions for these trees. Hasunuma \cite{2} introduced CIST and provided a first characterization:
\begin{thm}[\cite{2}]
    Let $T_1, \dots, T_k$ be spanning trees in a graph $G$. Then, $T_1, \dots , T_k$ are  completely independent if and only if $T_1, \dots , T_k$ are edge-disjoint and for any vertex $x \in V(G)$ and $1 \leq i \leq k$, there is at most one spanning tree $T_i$ such that $d_{T_i}(x) > 1.$
    \label{thm00}
\end{thm}

The same author \cite{12} also proved that the problem of determining the existence of two completely independent spanning trees in a graph is NP-Complete and conjectured that any $2k$-connected graph contains $k$ CIST. However, Péterfalvi \cite{4} proved that for any $k \geq 2$, there exists a $k$-connected graph that does not admit two CIST. Due to its complexity, several sufficient conditions have been investigated in order to guarantee the existence of $k$ CIST in graphs \cite{5,40,22,48,73,38,37}, most of them are derived from Hamiltonicity conditions. Additionally, various classes have been studied, including planar graphs \cite{12}, complete graphs \cite{31,6}, and $k$-trees \cite{61}. For more details, the reader may refer to \cite{82}. 

Another essential characterization for studying CIST that will be used in this paper has been provided by Araki~\cite{5}: Let $G$ be a graph, and let $V_1 \cup V_2 \cup \dots \cup V_k$ be a partition of its vertices. The partition $V(G) = V_1 \cup V_2 \cup \dots \cup V_k$ is a {\em $k$-CIST-partition} if $G[V_i]$ is connected for each $i= 1,2,\dots, k$, and the bipartite subgraph $B(V_i,V_j,G)$ of $G$ induced by $V_i \cup V_j$ has no tree component for each $1 \leq i < j \leq k $, \textit{i.e.}, every connected component $H$ of $B(V_i,V_j,G)$ satisfies $|E(H)| \geq |V(H)|$.
\begin{thm}[\cite{5}]\label{thm2}
    A graph $G$ admits $k$ completely independent spanning trees if and only if it has a $k$-CIST-partition.
\end{thm}
A {\em split graph}  $G = (D \cup I,E)$ is a graph in which the vertices can be partitioned into a clique $D$ and an independent set $I$. Chen \textit{et al.} \cite{29} gave sufficient conditions for a split graph to have two completely independent spanning trees:
\begin{thm}[\cite{29}]
    If $G$ is a Hamiltonian split graph such that $|D| > \max \{ 3,|I| \}$, then $G$ has two CIST.
    \label{thm3}
\end{thm}
Let $w(G)$ be the number of connected components of a graph $G$. Chen \textit{et al.} \cite{29} also proved that if $G$ is a Hamiltonian split graph and $\tau(G) >1$, then $G$ has two CIST, where $\tau(G) = \min \{ \frac{|S|}{w(G-S)} : w(G-S) \geq 2 , S \subset V\}$ is the {\em toughness} of $G$. 

In this paper, we continue the study of the CIST problem on split graphs and give a necessary condition relating the hypergraph panchromatic coloring problem and the CIST existence problem in this class of graphs. We also introduce a variant of the panchromatic coloring called bipanchromatic coloring to obtain sufficient conditions for CIST in split graphs. The previous conditions allow us to establish tight bounds on the maximum number of CIST. Moreover, we prove that the CIST problem in split graphs is NP-Complete and propose integer linear programming models that allow us to formulate a conjecture for the hypergraph panchromatic coloring problem. 

The rest of this paper is organized as follows. Section \ref{sec2} provides the preliminaries used throughout this paper, specifically the panchromatic and bipanchromatic coloring of hypergraphs, as well as the notion of a corresponding hypergraph to a split graph. Section \ref{sec3} presents results on the number of CIST in split graphs according to the panchromatic and bipanchromatic colorings of hypergraphs. Section \ref{sec4} proves that the CIST problem in split graphs is NP-complete. Section \ref{sec5} concludes the paper by giving some research directions. Finally, in Appendix \ref{appA}, we provide ILP models for the panchromatic and bipanchromatic coloring problems. 

\section{Preliminaries} \label{sec2}
A hypergraph $H = (V, \mathcal{E}) $ is a generalization of a graph in which an edge $e \in \mathcal{E}$, called a {\em hyperedge}, can join any number of vertices, \textit{i.e.} $\mathcal{E} \subset \mathcal{P}(V)$ where $\mathcal{P}(V)$ is the set of all subsets of $V$. It is {\em $k$-uniform} if every hyperedge contains exactly $k$ vertices. Given a split graph $G = (D \cup I,E)$, we define its {\em corresponding hypergraph} $H(G) = (D,\mathcal{E})$ with $\mathcal{E} = \{ N_G(x), x \in I\}$ (see Figure \ref{alpha}). Obviously, every split graph has exactly one corresponding hypergraph and vice versa. The corresponding split graph of a hypergraph $H$ is denoted $G(H)$, and the corresponding hyperedge in $H(G)$ of a vertex $x \in I$ is denoted by $e_x$. Without loss of generality, we assume that every vertex of $D$ is adjacent to at least one vertex of $I$; otherwise, a vertex in $D$ with no neighbor in $I$ can be reassigned to $I$ without changing the split structure.
 \begin{figure}[ht]
    \begin{center}
    \begin{tikzpicture}[scale=0.8]
        \node[rectangle,draw=black,fill=green] (01) at (-8,0) {};
        \node[regular polygon, regular polygon sides=3,draw=black,fill=blue,scale=0.5] (02) at (-6,-1) {};
        \node[shape=rectangle,draw=black,fill=green] (03) at (-4,0) {};
        \node[regular polygon, regular polygon sides=3,draw=black,fill=blue,scale=0.5] (04) at (-4,2) {};
        \node[shape=rectangle,draw=black,fill=green] (05) at (-6,3) {};
        \node[shape=circle,draw=black,fill=red] (06) at (-8,2) {};
        \node[shape=circle,draw=black,scale=0.7] (07) at (-6,4.2) {};
        \node[shape=circle,draw=black,scale=0.7] (08) at (-9.5,1) {};
        \node[shape=circle,draw=black,scale=0.7] (09) at (-2.8,1) {};
        \node[shape=circle,draw=black,scale=0.7] (010) at (-9,3) {};
        \node[rectangle,draw=black,fill=green] (1) at (0,0) {};
        \node[regular polygon, regular polygon sides=3,draw=black,fill=blue,scale=0.5] (2) at (2,-1) {};
        \node[shape=rectangle,draw=black,fill=green] (3) at (4,0) {};
        \node[regular polygon, regular polygon sides=3,draw=black,fill=blue,scale=0.5] (4) at (4,2) {};
        \node[shape=rectangle,draw=black,fill=green] (5) at (2,3) {};
        \node[shape=circle,draw=black,fill=red] (6) at (0,2) {};        
        \path [-] (01) edge node {} (04);
        \path [-] (01) edge node {} (02);
        \path [-] (01) edge node {} (06);
        \path [-] (04) edge node {} (03);
        \path [-] (04) edge node {} (05);
        \path [-] (01) edge node {} (03);
        \path [-] (01) edge node {} (05);
        \path [-] (05) edge node {} (03);    
        \path [-] (05) edge node {} (02);    
        \path [-] (06) edge node {} (02);     
        \path [-] (06) edge node {} (03);    
        \path [-] (02) edge node {} (03);
        \path [-] (05) edge node {} (06);
        \path [-] (06) edge node {} (04);    
        \path [-] (04) edge node {} (02); 
        \path [-] (07) edge node {} (04); 
        \path [-] (07) edge node {} (05); 
        \path [-] (07) edge node {} (06); 
        \path [-] (08) edge node {} (01); 
        \path [-] (08) edge node {} (06); 
        \path [-] (09) edge node {} (03); 
        \path [-] (09) edge node {} (06);
        \path [-] (010) edge node {} (01); 
        \path [-] (010) edge node {} (05); 
        \path [-] (010) edge node {} (06); 
        \draw[-] (02) to[out=180, in= -90] (08); 
        \draw[-] (02) to[out=0, in= -90] (09); 
        \begin{scope}[fill opacity=0]
        \filldraw[fill=white] ($(5)+(0,0.5)$)
            to[out=180,in=90] ($(6) + (-0.6,-0.05)$) 
            to[out=-90,in=180] ($(5) + (0,-0.5)$) 
            to[out=0,in=-90] ($(4) + (0.4,0)$)         
            to[out=90,in=0] ($(5) + (0,0.5)$)  ;
        \end{scope}    
        \begin{scope}[fill opacity=0]
        \filldraw[fill=white] ($(3)+(0,0.5)$)
            to[out=180,in=0] ($(2) + (0.6,0.5)$)  
            to[out=180,in=0] ($(6) + (0,0.4)$)  
            to[out=180,in=90] ($(6) + (-0.4,0)$)         
            to[out=-90,in=100] ($(2) + (-0.6,-0.1)$) 
            to[out=-90,in=180] ($(2) + (0,-0.5)$) 
            to[out=0,in=-90] ($(3) + (0.5,0)$)          
            to[out=90,in=0] ($(3) + (0,0.5)$)   ;
        \end{scope}    
        \begin{scope}[fill opacity=0]
        \filldraw[fill=white] ($(1)+(-0.8,0)$)
            to[out=-90,in=180] ($(2) + (0,-0.4)$)  
            to[out=0,in=0] ($(6) + (0.2,0.7)$)  
            to[out=180,in=90] ($(1) + (-0.8,0)$) ;
        \end{scope}    
        \begin{scope}[fill opacity=0]
        \filldraw[fill=white] ($(1)+(0,-0.4)$)
            to[out=0,in=-90] ($(6) + (0.5,0)$)  
            to[out=90,in=180] ($(5) + (0.3,-0.4)$)  
            to[out=0,in=0] ($(5) + (0.3,0.4)$)          
            to[out=180,in=90] ($(6) + (-0.5,0)$)
            to[out=-90,in=180] ($(1) + (0,-0.4)$)   ;
        \end{scope}       
    \end{tikzpicture}
    \caption{A split graph and its corresponding hypergraph}
    \label{alpha}
    \end{center}
 \end{figure}
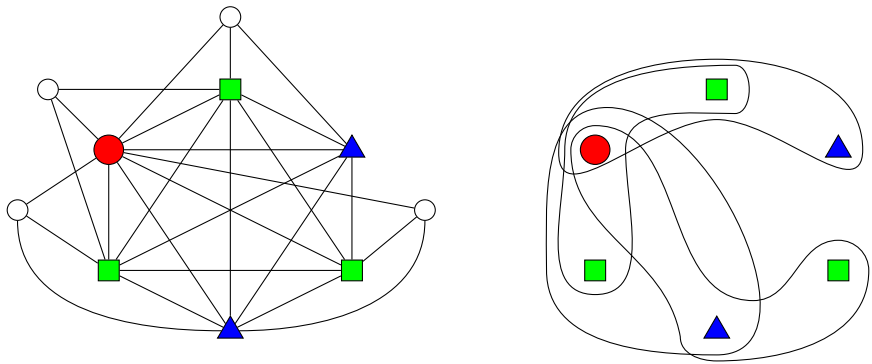

Hypergraph coloring extends the concept of graph coloring: A $k$-coloring of a hypergraph is an assignment of $k$ colors to its vertices so that no hyperedge is monochromatic. A {\em panchromatic} $k$-coloring of $H$ is a coloring of the vertices of $H$ such that each hyperedge contains at least one vertex of each color. A color that appears only once in the hypergraph will be called a {\em unique} color. The problem of deciding whether a hypergraph is $2$-colorable (known as the set splitting problem or Property B) is NP-Complete \cite{garey} and by extension for $k\geq 3$. Since a panchromatic $2$-coloring is a $2$-coloring, the problem of determining whether a hypergraph is panchromatically $2$-colorable is also NP-complete.

We define a variation of panchromatic coloring that will be useful for the CIST problem in split graphs:  If a panchromatic $k$-coloring $\varphi$ of a hypergraph $H$ is such that each of the $k$ colors appears at least twice in $H$ (\textit{i.e.}, $\varphi$ has no unique color), then $\varphi$ is called {\em bipanchromatic}. 
The panchromatic (bipanchromatic, respectively) number, denoted $\chi_p(H)$ ($\chi^2_p(H)$, respectively), is the maximum $k$ such that $H$ admits a panchromatic (bipanchromatic, respectively) $k$-coloring. We have trivially $\chi_p^2(H) \le \chi_p(H)$. But the inequality can be strict: The hypergraph on the right of Figure~\ref{alpha} is such that $\chi_p = 3$ and $\chi^2_p = 2$ (a bipanchromatic $2$-coloring is obtained by changing the color of the red vertex to blue). 

Most of the research works on panchromatic coloring are related to the study of the minimum number of hyperedges in an $n$-uniform hypergraph that does not admit a panchromatic $r$-coloring, denoted by $p(n,r)$. Many sufficient conditions and bounds for $p(n,r)$ have been proposed; for more details, the reader may refer to \cite{colorcritical,246}. 

\section{CIST and hypergraph coloring}\label{sec3}
In this section, we demonstrate a direct correspondence between CIST in split graphs and panchromatic colorings of their associated hypergraphs. We first prove that the existence of $k$ CIST in a split graph $G$ implies that its corresponding hypergraph $H(G)$ is panchromatically $k$-colorable. Then, to analyze the converse, we examine the links between CIST and the presence of colors that appear only once in a panchromatic $k$-coloring of $H(G)$. We conclude with an upper bound on the number of CIST.
\begin{thm}
    Let $k \geq 2$. Let $G= (D \cup I,E)$ be a split graph. If $G$ has $k$ completely independent spanning trees, then $H(G)$ is panchromatically $k$-colorable.
    \label{thm1}
\end{thm}
\begin{proof}
    Let $G= (D \cup I,E)$ be a split graph, and $T_1,T_2,\dots,T_k$ be completely independent spanning trees of $G$. 
    By Theorem~\ref{thm2}, $G$ admits a CIST-partition $V_1,V_2,\dots,V_k$. Let us construct a panchromatic $k$-coloring $\varphi$ of $H(G)$: Each vertex in $V_i$ is assigned the color $i$, $ 1 \leq i \leq k$.
    By the definition of a CIST-partition, every vertex $x\in I$ is adjacent to a vertex of each of the $k$ spanning trees. Then, the hyperedge $e_x$ contains all the $k$ colors of $\varphi$. Thus, $\varphi$ is a panchromatic $k$-coloring of $H(G)$.
\end{proof}

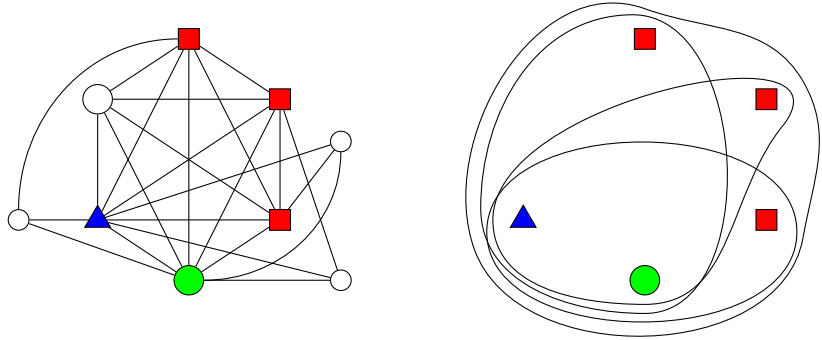
\begin{figure}[ht]
    \begin{center}
    \begin{tikzpicture}[scale=0.8]
        \node[shape=rectangle,draw=black,fill=red] (1) at (9,3) {};
        \node[regular polygon, regular polygon sides=3,draw=black,fill=blue,scale=0.5] (3) at (7,0) {};
        \node[shape=circle,draw=black,fill=green] (4) at (9,-1) {};
        \node[shape=rectangle,draw=black,fill=red] (6) at (11,0) {};
        \node[shape=rectangle,draw=black,fill=red] (5) at (11,2) {};
        \node[shape=rectangle,draw=black,fill=red] (01) at (1.5,3) {};
        \node[shape=circle,draw=black] (02) at (0,2) {};
        \node[regular polygon, regular polygon sides=3,draw=black,fill=blue,scale=0.5] (03) at (0,0) {};
        \node[shape=circle,draw=black,fill=green] (04) at (1.5,-1) {};
        \node[shape=rectangle,draw=black,fill=red] (06) at (3,0) {};
        \node[shape=rectangle,draw=black,fill=red] (05) at (3,2) {};
        \node[shape=circle,draw=black,scale=0.7] (07) at (4,1.3) {};
        \node[shape=circle,draw=black,scale=0.7] (08) at (4,-1) {};
        \node[shape=circle,draw=black,scale=0.7] (09) at (-1.3,0) {};   
        \path [-] (01) edge node {} (04);
        \path [-] (01) edge node {} (02);
        \path [-] (01) edge node {} (06);
        \path [-] (04) edge node {} (03);
        \path [-] (04) edge node {} (05);
        \path [-] (01) edge node {} (03);
        \path [-] (01) edge node {} (05);
        \path [-] (05) edge node {} (03);    
        \path [-] (05) edge node {} (02);    
        \path [-] (06) edge node {} (02);     
        \path [-] (06) edge node {} (03);    
        \path [-] (02) edge node {} (03);
        \path [-] (05) edge node {} (06);
        \path [-] (06) edge node {} (04);    
        \path [-] (04) edge node {} (02);    
        \path [-] (07) edge node {} (06);
        \path [-] (07) edge node {} (03);    
        \path [-] (08) edge node {} (05);
        \path [-] (08) edge node {} (04);    
        \path [-] (08) edge node {} (03);
        \path [-] (09) edge node {} (04);    
        \path [-] (09) edge node {} (03);
        \draw[-] (09) to[out=90, in= 180] (01); 
        \draw[-] (07) to[out=-90, in= 0] (04); 
        \begin{scope}[fill opacity=0]
        \filldraw[fill=yellow!70] ($(5)+(0.3,-0.4)$)
            to[out=-130,in=0] ($(4) + (0,-0.4)$) 
            to[out=180,in=-90] ($(3) + (-0.5,0)$) 
            to[out=90,in=55] ($(5)+(0.3,-0.4)$)  ;
        \end{scope}
        \begin{scope}[fill opacity=0]
        \filldraw[fill=yellow!70] ($(6)+(0.5,-0.2)$)
            to[out=-90,in=-90] ($(3) + (-0.6,-0.2)$)  
            to[out=90,in=90] ($(6) + (0.5,-0.2)$) ;
        \end{scope}
        \begin{scope}[fill opacity=0]
        \filldraw[fill=yellow!70] ($(1)+(-0.2,0.4)$)
            to[out=0,in=0] ($(4) + (0,-0.55)$)
            to[out=180,in=-90] ($(3) + (-0.7,0.2)$)
            to[out=90,in=180] ($(1) + (-0.2,0.4)$) ;
        \end{scope}
        \begin{scope}[fill opacity=0]s
        \filldraw[fill=yellow!70] ($(1)+(0,0.5)$)
            to[out=-20,in=120] ($(5) + (0.6,0.3)$)
            to[out=-60,in=80] ($(6) + (0.6,-0.3)$)    
            to[out=-100,in=-70] ($(3) + (-0.8,-0.5)$)        
            to[out=110,in=160] ($(1) + (0,0.5)$) ;
        \end{scope}   
    \end{tikzpicture}
        \caption{A split graph without $3$ CIST whose corresponding hypergraph is panchromatically $3$-colorable }
        \label{ref1}
    \end{center}
\end{figure}

\begin{obs}
    Péterfalvi's counterexample \cite{4} to Hasunuma's conjecture is a split graph whose corresponding hypergraph is not $k$-colorable for $k \geq 2$. Using the contrapositive of Theorem \ref{thm1}, this also implies the nonexistence of $k$ CIST. 
\end{obs}
However, the converse is not always true. Figure \ref{ref1} shows a split graph that does not admit three completely independent trees (this can be verified using a linear program \cite{34}), yet its corresponding hypergraph is panchromatically $3$-colorable. In the following proposition, additional conditions are added to get the converse. 

\begin{thm}
    Let $k \geq 2$. Let $G$ be a split graph. If $H(G)$ is bipanchromatically $k$-colorable, then $G$ has $k$ completely independent spanning trees.
    \label{lemme1}
\end{thm}
\begin{proof}
    Let $G = (D \cup I,E)$ be a split graph. Let us begin with the construction of a $k$-CIST-partition for the subgraph induced by $D$: Each vertex of color $i$ for $i = 1,2,\dots,k$ is assigned to the set of internal vertices of the tree $T_i$ denoted $V_i$. From the result of Pai \textit{et al.} \cite{3}, only two vertices in $K_n$ suffice to construct each tree for each $n \geq 4$. As $H(G)$ is bipanchromatically $k$-colorable, the number of internal vertices per tree is at least two. Since $T_1,T_2,\dots,T_k$ are CIST in $D$, by Theorem \ref{thm2}, the partition $V_1 \cup V_2 \cup \dots \cup V_k$ is $k$-CIST partition in $D$. Now, let us extend the previous partition to the vertices of $I$: Since every vertex of $I$ in $G$ is covered by an internal vertex of each tree $T_i$, then the constructed trees are spanning. Moreover, if each vertex of $I$ is assigned arbitrarily to one of the previous sets, then the new sets,  denoted $V'_i$, have the following properties: $G[V'_i]$ is connected for each $i = 1,2,\dots,k$ because the constructed trees are spanning, and for each $1 \le i < j \le k$, $B(V'_i,V'_j,G)$ has no tree component because each element of $I$ is only adjacent to a component that is not a tree component. Thus, $V'_1,V'_2,\dots,V'_k$ form a CIST-partition in $G$.
\end{proof}

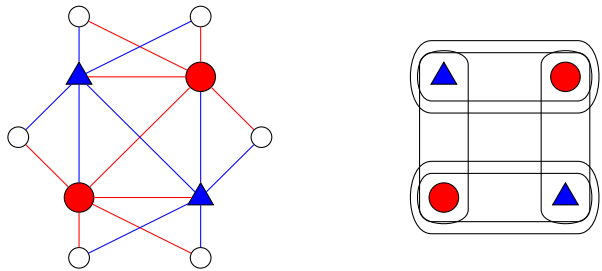
\begin{figure}[ht]
    \begin{center}
    \begin{tikzpicture}[scale=0.8]
        \node[shape=circle,draw=black,fill=red] (1) at (0,0) {};
        \node[regular polygon, regular polygon sides=3,draw=black,fill=blue,scale=0.5] (2) at (2,0) {};
        \node[shape=circle,draw=black,fill=red] (3) at (2,2) {};
        \node[regular polygon, regular polygon sides=3,draw=black,fill=blue,scale=0.5] (4) at (0,2) {};
        \node[shape=circle,draw=black,scale=0.7] (5) at (0,3) {};
        \node[shape=circle,draw=black,scale=0.7] (6) at (2,3) {};
        \node[shape=circle,draw=black,scale=0.7] (7) at (2,-1) {};
        \node[shape=circle,draw=black,scale=0.7] (8) at (0,-1) {};
        \node[shape=circle,draw=black,scale=0.7] (10) at (-1,1) {};
        \node[shape=circle,draw=black,scale=0.7] (11) at (3,1) {};
        \path [-,red] (1) edge node {} (3);    
        \path [-,red] (1) edge node {} (2);    
        \path [-,blue] (1) edge node {} (4);     
        \path [-,blue] (2) edge node {} (3);    
        \path [-,blue] (2) edge node {} (4);
        \path [-,red] (3) edge node {} (4);    
        \path [-,red] (5) edge node {} (3);
        \path [-,blue] (5) edge node {} (4);
        \path [-,red] (6) edge node {} (3);    
        \path [-,blue] (6) edge node {} (4);    
        \path [-,red] (7) edge node {} (1);    
        \path [-,blue] (7) edge node {} (2);    
        \path [-,red] (8) edge node {} (1);
        \path [-,blue] (8) edge node {} (2);
        \path [-,red] (10) edge node {} (1);
        \path [-,blue] (10) edge node {} (4);  
        \path [-,blue] (11) edge node {} (2);
        \path [-,red] (11) edge node {} (3);    
        \node[shape=circle,draw=black,fill=red] (01) at (6,0) {};
        \node[regular polygon, regular polygon sides=3,draw=black,fill=blue,scale=0.5] (02) at (8,0) {};
        \node[shape=circle,draw=black,fill=red] (03) at (8,2) {};
        \node[regular polygon, regular polygon sides=3,draw=black,fill=blue,scale=0.5] (04) at (6,2) {};
        \begin{scope}[fill opacity=0]
        \filldraw[fill=yellow!70] ($(03)+(0.4,0.2)$)
            to[out=-90,in=90] ($(02) + (0.4,-0.2)$)
            to[out=-90,in=-90] ($(02) + (-0.4,-0.2)$) 
            to[out=90,in=-90] ($(03) + (-0.4,0.2)$)  
            to[out=90,in=90] ($(03) + (0.4,0.2)$)   ;
        \end{scope}
        \begin{scope}[fill opacity=0]
        \filldraw[fill=yellow!70] ($(04)+(0.4,0.2)$)
            to[out=-90,in=90] ($(01) + (0.4,-0.2)$)
            to[out=-90,in=-90] ($(01) + (-0.4,-0.2)$) 
            to[out=90,in=-90] ($(04) + (-0.4,0.2)$)  
            to[out=90,in=90] ($(04) + (0.4,0.2)$)   ;
        \end{scope}    
        \begin{scope}[fill opacity=0]
        \filldraw[fill=yellow!70] ($(04)+(-0.2,0.4)$)
            to[out=0,in=180] ($(03) + (0.2,0.4)$)
            to[out=0,in=0] ($(03) + (0.2,-0.4)$) 
            to[out=180,in=0] ($(04) + (-0.2,-0.4)$)  
            to[out=180,in=180] ($(04) + (-0.2,0.4)$)   ;
        \end{scope}   
        \begin{scope}[fill opacity=0]
        \filldraw[fill=yellow!70] ($(01)+(-0.2,0.4)$)
            to[out=0,in=180] ($(02) + (0.2,0.4)$)
            to[out=0,in=0] ($(02) + (0.2,-0.4)$) 
            to[out=180,in=0] ($(01) + (-0.2,-0.4)$)  
            to[out=180,in=180] ($(01) + (-0.2,0.4)$)   ;
        \end{scope}   
        \begin{scope}[fill opacity=0]
        \filldraw[fill=yellow!70] ($(04)+(-0.2,0.6)$)
            to[out=0,in=180] ($(03) + (0.2,0.6)$)
            to[out=0,in=0] ($(03) + (0.2,-0.6)$) 
            to[out=180,in=0] ($(04) + (-0.2,-0.6)$)  
            to[out=180,in=180] ($(04) + (-0.2,0.6)$)   ;
        \end{scope}   
        \begin{scope}[fill opacity=0]
        \filldraw[fill=yellow!70] ($(01)+(-0.2,0.6)$)
            to[out=0,in=180] ($(02) + (0.2,0.6)$)
            to[out=0,in=0] ($(02) + (0.2,-0.6)$) 
            to[out=180,in=0] ($(01) + (-0.2,-0.6)$)  
            to[out=180,in=180] ($(01) + (-0.2,0.6)$)   ;
        \end{scope}   
    \end{tikzpicture}
        \caption{A split graph having $2$ CIST while not satisfying the conditions of Theorem \ref{thm3}}
        \label{figurenv}
    \end{center}
\end{figure}

\noindent Theorem \ref{lemme1} provides new results that were not implied by Theorem \ref{thm3}. In fact, it shows that even some graphs that are not Hamiltonian and do not satisfy $|D| > \max \{ 3,|I| \}$ can have two CIST (see Figure~\ref{figurenv}). However, Theorem \ref{lemme1} relies on the existence of a bipanchromatic $k$-coloring, which is not always necessary. Figure \ref{ref2} shows a graph with two completely independent spanning trees, yet its corresponding hypergraph is not bipanchromatically $2$-colorable.
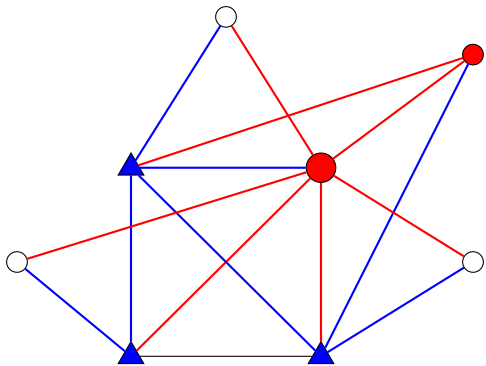
\begin{figure}[ht]
    \begin{center}
    \begin{tikzpicture}
        \node[regular polygon, regular polygon sides=3,draw=black,fill=blue,scale=0.5] (1) at (0,0) {};
        \node[regular polygon, regular polygon sides=3,draw=black,fill=blue,scale=0.5] (2) at (2.5,0) {};
        \node[shape=circle,draw=black,fill=red] (3) at (2.5,2.5) {};
        \node[regular polygon, regular polygon sides=3,draw=black,fill=blue,scale=0.5] (4) at (0,2.5) {};
        \node[shape=circle,draw=black,fill=red,scale=0.7] (8) at (4.5,4) {};
        \node[shape=circle,draw=black,scale=0.7] (10) at (-1.5,1.25) {};
        \node[shape=circle,draw=black,scale=0.7] (11) at (4.5,1.25) {};
        \node[shape=circle,draw=black,scale=0.7] (12) at (1.25,4.5) {};
        \path [-,red,thick] (1) edge node {} (3);    
        \path [-] (1) edge node {} (2);    
        \path [-,blue,thick] (1) edge node {} (4);     
        \path [-,red,thick] (2) edge node {} (3);    
        \path [-,blue,thick] (2) edge node {} (4);
        \path [-,blue,thick] (3) edge node {} (4);
        \path [-,red,thick] (8) edge node {} (3);
        \path [-,red,thick] (8) edge node {} (4);
        \path [-,blue,thick] (8) edge node {} (2);
        \path [-,blue,thick] (1) edge node {} (10);    
        \path [-,red,thick] (10) edge node {} (3);    
        \path [-,red,thick] (11) edge node {} (3);    
        \path [-,blue,thick] (11) edge node {} (2);      
        \path [-,red,thick] (12) edge node {} (3);    
        \path [-,blue,thick] (12) edge node {} (4);    
    \end{tikzpicture}
    \caption{A split graph with $2$ CIST whose corresponding hypergraph is not bipanchromatically $2$-colorable}
    \label{ref2}
    \end{center}
\end{figure}

In what follows, some results are given for the case where, given $k\ge 2$, at least one vertex of unique color exists in any panchromatic $k$-coloring of $H$. To this end, we let  $\alpha_k(H(G))$ be the minimum number of vertices of unique color among all panchromatic $k$-colorings of $H(G)$. 
\subsection{Properties related to unique colors}
\begin{lemme}
    Let $k \geq 2$. Let $G=(D\cup I,E)$ be a split graph. Let $x_D$ be a vertex of $D$ and $\varphi$ be a panchromatic $k$-coloring of $H(G)$. If the color of $x_D$ is unique under $\varphi$, then $x_D$ is adjacent to all vertices of $I$.
    \label{lemme2}
\end{lemme}
\begin{proof}
    Suppose that the color of $x_D$ is unique and that $x_D$ is not adjacent to all the vertices of $I$. Therefore, there exists $y \in I$ that is not adjacent to $x_D$. Then, $e_y$ does not contain the color of $x_D$. Thus, $\varphi$ is not panchromatic, a contradiction.
\end{proof}
    
\noindent However, Figure \ref{figcontre} shows that the converse is not always true. But, it holds if we restrict to split graphs such that $H(G)$ is $k$-uniform (\textit{i.e.}, all vertices of $I$ have a degree equal to $k$).

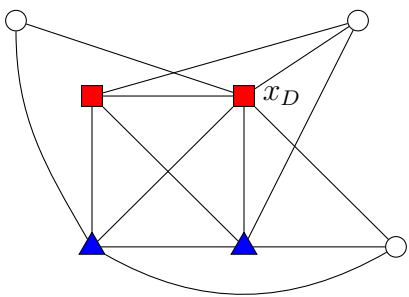
\begin{figure}[ht]
    \begin{center}
    \begin{tikzpicture}
        \node[regular polygon, regular polygon sides=3,draw=black,fill=blue,scale=0.5] (1) at (1,1) {};
        \node[regular polygon, regular polygon sides=3,draw=black,fill=blue,scale=0.5] (2) at (3,1) {};
        \node[shape=rectangle,draw=black,fill=red] (3) at (3,3) {};
        \node[shape=rectangle,draw=black,fill=red] (4) at (1,3) {};
        \node[shape=circle,draw=black,scale=0.7] (7) at (0,4) {};
        \node[shape=circle,draw=black,scale=0.7] (8) at (4.5,4) {};
        \node[shape=circle,draw=black,scale=0.7] (9) at (5,1) {};
        \path [-] (1) edge node {} (3);    
        \path [-] (1) edge node {} (2);    
        \path [-] (1) edge node {} (4);     
        \path [-] (2) edge node {} (3);    
        \path [-] (2) edge node {} (4);
        \path [-] (3) edge node {} (4);
        \path [-] (7) edge node {} (3);    
        \path [-] (8) edge node {} (3);
        \path [-] (8) edge node {} (4);
        \path [-] (8) edge node {} (2);
        \path [-] (9) edge node {} (2);
        \path [-] (9) edge node {} (3);
        \draw[-] (1) to[out=120, in= -90] (7); 
        \draw[-] (9) to[out=-150, in= -30] (1); 
        \node at (3.5,3) {$x_D$};
    \end{tikzpicture}
    \caption{A split graph $G$ where $x_D$ is adjacent to all vertices of $I$ but its color is not unique in some of the panchromatic $2$-colorings of $H(G)$}
    \label{figcontre}
    \end{center}
\end{figure}
    
\begin{lemme}
    Let $k \geq 2$. Let $G=(D\cup I,E)$ be a split graph such that $H(G)$ is $k$-uniform. Let $x_D$ be a vertex of $D$ and $\varphi$ be a panchromatic $k$-coloring of $H(G)$. If $x_D$ is adjacent to all vertices of $I$, then the color of $x_D$ is unique under $\varphi$.
    \label{lemme2a}
\end{lemme}
\begin{proof}
    Since the hypergraph $H(G)$ is $k$-uniform and panchromatically $k$-colorable, each hyperedge in $H(G)$ contains exactly $k$ vertices that have different colors. It follows that the vertex $x_D$ belongs to all the hyperedges and does not have the same color as the other vertices.
\end{proof}

\subsection{Relation between CIST and unique colors}
\begin{prop}
    Let $k \geq 2$. Let $G = (D \cup I,E)$ be a split graph such that its corresponding hypergraph $H(G)$ is $k$-uniform. If there exists $x_D \in D$ such that $I \subset N_G(x_D)$, then $G$ does not have $k$ completely independent spanning trees.
    \label{lemme3}
\end{prop}
\begin{proof}
    Suppose that $G$ has $k$ completely independent spanning trees. Each vertex $y \in I$ in $G$ is adjacent to $k$ vertices of $D$, including $x_D$. Then, $y$ is a leaf for all CIST of $G$. Without loss of generality, suppose $x_D$ is an internal vertex of the spanning tree $T_1$. By Lemma \ref{lemme2a}, $x_D$ is a vertex of unique color, thus it needs to cover all the vertices of $I$ in $T_1$. However, $x_D$ cannot cover all vertices of $D$ in $T_1$ because it has to be covered by an internal vertex of each of the $k-1$ other CIST. If $x_D$ is covered by $k-1$ internal vertices, each one corresponding to a spanning tree, then it cannot cover them too. It cannot also be covered by the vertices of $I$ in any tree because they are leaves. Therefore, $T_1$ is not a spanning tree, a contradiction (see Figure~\ref{figcontre1}).
\end{proof}

\begin{figure}[ht]
    \begin{center}
    \begin{tikzpicture}
        \node[shape=circle,draw=black,fill=blue] (0) at (1,-1.2) {};
        \node[shape=circle,draw=black,fill=red] (1) at (2,0) {};
        \node[shape=circle,draw=black,fill=green] (2) at (3.5,-0.7) {};
        \node[shape=circle,draw=black,fill=blue] (3) at (4.6,-2) {};
        \node[shape=circle,draw=black,fill=green] (4) at (3.7,-3) {};
        \node[shape=circle,draw=black,scale=0.7] (5) at (5,0) {};
        \node[shape=circle,draw=black,scale=0.7] (6) at (3.8,1.6) {};
        \path [-,thick,red] (1) edge node {} (3);    
        \path [-,thick,red] (1) edge node {} (2);    
        \path [-,thick,green] (1) edge node {} (4);     
        \path [-,thick,blue] (2) edge node {} (3);    
        \path [-,thick,green] (2) edge node {} (4);
        \path [-,thick,green] (3) edge node {} (4);
        \path [-,thick,red] (1) edge node {} (5);     
        \path [-,thick,green] (2) edge node {} (5);    
        \path [-,thick,blue] (3) edge node {} (5);
        \node at (1.55,0.25) {$x_D$};
        \node[] at (5.5,2) {$T_1$};
        \path [-,thick,red] (6,2) edge node {} (6.5,2);    
        \path [-,thick,blue] (1) edge node {} (0);     
        \path [-,thick,green] (2) edge node {} (0);    
        \path [-,thick,blue] (3) edge node {} (0);
        \path [-,thick,blue] (4) edge node {} (0);
        \path [-,thick,red] (1) edge node {} (6);    
        \path [-,thick,green] (2) edge node {} (6);
        \path [-,thick,blue] (3) edge node {} (6);
        \node[scale=1.3] at (1.75,-2.5) {$\ddots$};
        \path [-] (2.4,-3) edge node {} (4);    
        \path [-] (0) edge node {} (1.2,-2.3);    
    \end{tikzpicture}
    \end{center}
    \caption{Configuration of proof of Proposition \ref{lemme3} for $k = 3$ that leads to a contradiction}
    \label{figcontre1}
\end{figure}
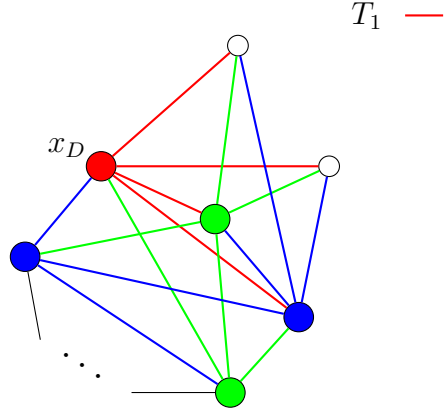
By Lemma \ref{lemme2a} and Proposition \ref{lemme2}, we have:
\begin{corollary}
    Let $G = (D \cup I,E)$ be a split graph such that its corresponding hypergraph $H(G)$ is $k$-uniform and panchromatically $k$-colorable. If $\alpha_k(H(G)) \geq 1$, then $G$ does not have $k$ completely independent spanning trees.
\end{corollary}

\noindent The $k$-uniformity of the hypergraph forbids the vertex $x_D$ to cover all vertices of $D$ through the vertices of $I$. However, Figure \ref{contr1} shows that the non-uniformity alone is not sufficient to have $k$ completely independent trees in a graph for which the associated hypergraph contains a $k$-uniform subhypergraph (the non-existence of $3$ CIST in this graph can be verified using a linear program \cite{34}). The following proposition shows that it is sufficient with an additional condition.

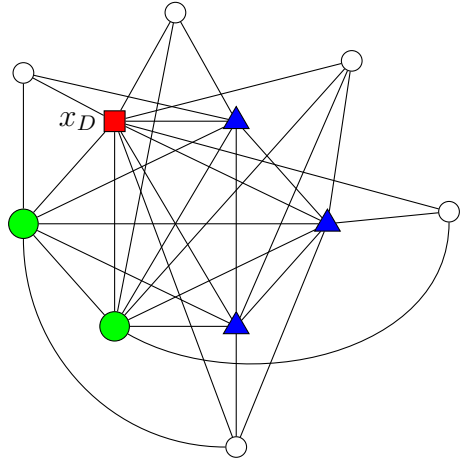
\begin{figure}[ht]
    \begin{center}
    \begin{tikzpicture}[scale=0.8]
        \node[shape=circle,draw=black,fill=green] (1) at (-0.5,0.5) {};
        \node[shape=circle,draw=black,fill=green] (2) at (1,-1.2) {};
        \node[regular polygon, regular polygon sides=3,draw=black,fill=blue,scale=0.5] (3) at (3,-1.2) {};
        \node[regular polygon, regular polygon sides=3,draw=black,fill=blue,scale=0.5] (4) at (4.5,0.5) {};
        \node[regular polygon, regular polygon sides=3,draw=black,fill=blue,scale=0.5] (6) at (3,2.2) {};
        \node[shape=rectangle,draw=black,fill=red] (5) at (1,2.2) {};
        \node[shape=circle,draw=black,scale=0.7] (7) at (4.9,3.2) {};
        \node[shape=circle,draw=black,scale=0.7] (8) at (6.5,0.7) {};
        \node[shape=circle,draw=black,scale=0.7] (9) at (3,-3.2) {};
        \node[shape=circle,draw=black,scale=0.7] (10) at (2,4) {};
        \node[shape=circle,draw=black,scale=0.7] (11) at (-0.5,3) {};
        \path [-] (1) edge node {} (4);
        \path [-] (1) edge node {} (2);
        \path [-] (1) edge node {} (6);
        \path [-] (4) edge node {} (3);
        \path [-] (4) edge node {} (5);
        \path [-] (1) edge node {} (3);
        \path [-] (1) edge node {} (5);
        \path [-] (5) edge node {} (3);    
        \path [-] (5) edge node {} (2);    
        \path [-] (9) edge node {} (5);    
        \path [-] (6) edge node {} (2);     
        \path [-] (6) edge node {} (3);    
        \path [-] (2) edge node {} (3);
        \path [-] (5) edge node {} (6);
        \path [-] (6) edge node {} (4);    
        \path [-] (4) edge node {} (2);    
        \path [-] (7) edge node {} (5);
        \path [-] (7) edge node {} (4);    
        \path [-] (7) edge node {} (3);    
        \path [-] (8) edge node {} (5);
        \path [-] (8) edge node {} (4);    
        \path [-] (9) edge node {} (4);    
        \path [-] (9) edge node {} (3);
        \draw[-] (9) to[out=180, in= -90] (1); 
        \draw[-] (8) to[out=-90, in= -30] (2); 
        \path [-] (7) edge node {} (2);
        \path [-] (5) edge node {} (10);        
        \path [-] (6) edge node {} (10);        
        \path [-] (2) edge node {} (10);
        \path [-] (1) edge node {} (11);        
        \path [-] (5) edge node {} (11);        
        \path [-] (6) edge node {} (11);
        \node at (0.4,2.2) {$x_D$};
    \end{tikzpicture}
    \caption{A split graph without $3$ CIST whose corresponding hypergraph is not $3$-uniform 
    }
    \label{contr1}
    \end{center}
\end{figure}
    
\begin{prop}
    Let $G = (D \cup I,E)$ be a split graph such that its corresponding hypergraph $H(G)$ is not $k$-uniform. If there exists a panchromatic $k$-coloring $\varphi$ of $H(G)$ such that there is exactly one vertex of unique color $c_j$, and for each color $c_i \neq c_j $, there exists a vertex $y \in I$ such that $d(y) > k$ and $y$ is adjacent to at least two vertices of color $c_i$, then $G$ has $k$ completely independent spanning trees.      
    \label{lemme5}
\end{prop}    
\begin{proof}
    Let $x_D$ be the vertex of unique color. Since $\varphi$ is panchromatic, every edge contains $k$ colors. The $k$ completely independent spanning trees $T_1,T_2,\dots,T_k$ are constructed as follows: Each vertex of color $c_i$ is an internal vertex of the tree $T_i$ in $G$. As previously shown in the proof of Theorem \ref{lemme1}, the subgraph induced by $D - \{ x_D \}$ contains $k-1$ completely independent spanning trees.
    Without loss of generality, let $T_j$ be the tree with $x_D$ its unique internal vertex in $D$. Let $T_i$ be another spanning tree and $x_{i_1}, x_{i_2}$ two of its internal vertices. Let $y\in I$ be the internal vertex of $T_j$ in $I$ that is adjacent to $ x_{i_1}$ and $ x_{i_2}$. Since $y$ is a leaf of $T_i$, it can be covered only by one internal vertex $x_{i_1}$ of $T_i$. It follows that the vertex $x_D$ can cover $x_{i_1}$ and is covered by $x_{i_2}$. Also, $x_D$ covers all internal vertices of $T_i$ non-adjacent to $y$, and $y$ covers the remaining internal vertices of $T_i$ (see Figure \ref{fig6}). Therefore, $T_j$ is a spanning tree and $T_1,T_2,\dots,T_k$ are $k$ completely independent spanning trees.
\end{proof}

\begin{figure}[ht]
    \begin{center}
    \begin{tikzpicture}
        \node[regular polygon, regular polygon sides=3,draw=black,fill=blue,scale=0.5] (1) at (2,1) {};
        \node[regular polygon, regular polygon sides=3,draw=black,fill=blue,scale=0.5] (2) at (4.5,1.7) {};
        \node[regular polygon, regular polygon sides=3,draw=black,fill=blue,scale=0.5] (3) at (3,3) {};
        \node[shape=circle,draw=black,fill=red] (4) at (1,3) {};
        \node[shape=circle,draw=black,fill=red,scale=0.7] (7) at (2,5) {};
        \path [-,red,thick] (4) edge node {} (7);    
        \path [-,blue,thick] (1) edge node {} (2);    
        \path [-,red,thick] (1) edge node {} (4);     
        \path [-,blue,thick] (2) edge node {} (3);    
        \path [-,blue,thick] (2) edge node {} (4);
        \path [-,red,thick] (3) edge node {} (4);
        \path [-,blue,thick] (7) edge node {} (3);    
        \draw[-,red,thick] (2) to[out=90, in= 0] (7); 
        \path [-,thick,red] (6.5,4.5) edge node {} (6,4.5);  
        \path [-,thick,blue] (6.5,5) edge node {} (6,5);      
        \node[] at (5.5,4.5) {$T_j$};
        \node[] at (5.5,5) {$T_i$};
        \node at (0.5,3.2) {$x_D$};
        \node at (3.5,3.2) {$x_{i_1}$};
        \node at (5,1.7) {$x_{i_2}$};
        \node at (1.7,5.2) {$y$};
    \end{tikzpicture}
    \caption{Illustration of the proof of Proposition~\ref{lemme5}}
    \label{fig6}
    \end{center}
\end{figure}
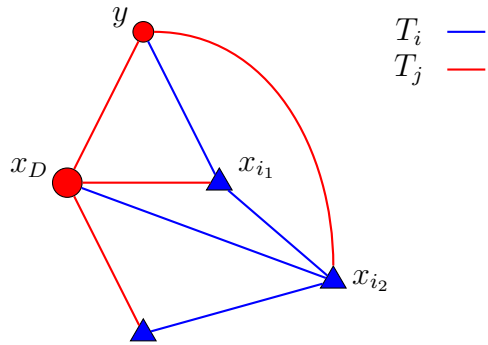
\begin{obs}
    If at least two unique colors are present in a panchromatic $k$-coloring of $H$, then all of their corresponding vertices can be grouped two by two until at most one vertex of unique color remains. By "grouping", we mean that, without loss of generality, for vertices $x_i$ and $x_j$ with unique colors $c_i$ and $c_j$, respectively, we change the color of $x_i$ to $c_j$ (hence color $c_j$ is no longer a unique color and color $c_i$ disappears). 
    \label{obs2}
\end{obs}
\begin{thm}
   Let $G = (D \cup I,E)$ be a split graph. If $H(G)$ is panchromatically $k$-colorable, then $G$ has $(k - \lceil \frac{\alpha_k(H(G))}{2} \rceil)$ completely independent spanning trees. 

   Moreover, if $H(G)$ is not $k$-uniform and there is a panchromatic $k$-coloring of $H(G)$ for which there is only one unique color $c_j$ and for each color $c_i \neq c_j $, there exists a vertex $y \in I$ such that $d(y) > k$ and $y$ is adjacent to at least two vertices of color $c_i$, then $G$ has $(k - \lceil \frac{\alpha_k(H(G))}{2} \rceil + 1)$ completely independent spanning trees. 
   \label{thm4}
\end{thm}
\begin{proof}
    Let $\varphi$ be a panchromatic $k$-coloring of $H(G)$ with the minimum number of unique colors $\alpha = \alpha_k(H(G))$. According to Observation \ref{obs2}, pairing up the vertices of unique colors two by two, the remaining is at most one. If such a vertex remains, let us change its color to any other color already assigned (to at least two vertices), resulting in a bipanchromatic $(k - \lceil \frac{\alpha}{2} \rceil)$-coloring. Therefore, by Theorem \ref{lemme1}, $G$ has $(k -  \lceil \frac{\alpha}{2} \rceil)$ completely independent spanning trees.
    Moreover, according to Proposition \ref{lemme5}, $G$ has $(k - \lceil \frac{\alpha}{2} \rceil + 1)$ completely independent spanning trees if $H(G)$ is not uniform, and for the vertex $x_D$ of unique color $c_j$ and each color $c_i \neq c_j$, there exists a vertex $y \in I$ such that $d(y) > k$ and $y$ is adjacent to at least two vertices of color $c_i$. 
\end{proof}

\subsection{Bounds on the number of CIST in split graphs}
\begin{thm}
   Let $G = (D \cup I,E)$ be a split graph and let $k = \chi_p^2 (H(G))$.  $G$ does not have $k+ 2$ completely independent spanning trees. 
   \label{thm7}
\end{thm}
\begin{proof}
    Suppose that $G$ has $k=\chi_p^2(H(G)) + 2$ completely independent spanning trees. The following cases are considered:
    \begin{itemize}
        \item Case 1: At most one tree has a single internal vertex in $D$.\\
        Let $t \in \{ 0,1 \}$ be the number of trees that have a single internal vertex in $D$. If $t=0$, then all internal vertices of each spanning tree belong to $D$. Hence, each vertex of $I$ is adjacent to $k+2$ internal vertices of different trees. According to the construction of the proof of Theorem \ref{lemme1}, each color appears at least twice in $H(G)$. Otherwise, if $t=1$, then only one tree has one unique internal vertex in $D$. Consequently, each vertex of $I$ is adjacent to $k + 1$ internal vertices of different trees. Reassigning that unique vertex to another color, the construction of Theorem \ref{lemme1} shows that each color appears at least twice in $H(G)$. Thus, for each value of $t$, $H(G)$ is at least bipanchromatically $(k + 1)$-colorable, a contradiction. 
        \item Case 2: At least two trees have a single internal vertex in $D$.\\
        Without loss of generality, let $T_1$ and $T_2$ be the considered trees and $x_1, x_2$ their unique internal vertex in $D$, respectively. In the best case, these vertices are both adjacent to at least two common vertices in $I$, denoted $i_1$ and $i_2$. However, if $x_1$ covers $x_2$ in $T_1$, $x_2$ cannot cover $x_1$ in $T_2$. Also, even if $i_1$ and $i_2$ are internal vertices of $T_2$, they cannot cover $x_1$ in $T_2$ because they are already covered by $x_1$ in $T_1$. Thus, $T_2$ is not a spanning tree, a contradiction. 
        Another way to illustrate the contradiction is through the number of edges:
        The trees $T_1$ and $T_2$ require six edges to cover $x_1,x_2,i_1, i_2$, but the actual number of edges is five, as shown in Figure \ref{contr}.
    \end{itemize}
\end{proof}
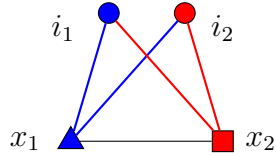
\begin{figure}[ht]
    \begin{center}
    \begin{tikzpicture}
        \node[regular polygon, regular polygon sides=3,draw=black,fill=blue,scale=0.5] (1) at (0.5,-0.5) {};
        \node[shape=rectangle,draw=black,fill=red] (2) at (2.5,-0.5) {};
        \node[circle,draw=black,fill=blue,scale=0.7] (3) at (1,1.2) {};
        \node[shape=circle,draw=black,fill=red,scale=0.7] (4) at (2,1.2) {};
        \path [-,thick,blue] (1) edge node {} (3);    
        \path [-] (1) edge node {} (2);    
        \path [-,thick,blue] (1) edge node {} (4);     
        \path [-,thick,red] (2) edge node {} (3);    
        \path [-,thick,red] (2) edge node {} (4);
        \node at (-0.1,-0.5) {$x_1$};
        \node at (3,-0.5) {$x_2$};
        \node at (0.4,1) {$i_1$};       
        \node at (2.5,1) {$i_2$};    
    \end{tikzpicture}
    \caption{Configuration of proof of Theorem \ref{thm7} leading to a contradiction (Case~2)}
    \label{contr}
    \end{center}
\end{figure}
The following theorem summarizes the previous results.
\begin{thm}
    Let $G = (D \cup I,E)$ be a split graph and $H(G)$ its corresponding hypergraph. The maximum number $M$ of CIST of $G$ satisfies: $$\chi_p^2(H(G)) \leq M \leq \chi_p^2(H(G))+ 1$$
    \label{prop}
\end{thm}

\section{Complexity}\label{sec4}
In this section, we investigate the computational complexity of the two problems considered in this paper, namely the bipanchromatic $k$-coloring problem in hypergraphs and the $k$ CIST problem in split graphs, where $k \geq 2$. First, we start with the bipanchromatic $k$-coloring problem, whose recognition version is defined as follows: 

\begin{table}[H]
\raggedright
\begin{tabular}{p{2cm} p{10cm}}
\multicolumn{2}{l}{\textsc{\textbf{Bipanchromatic $k$-Coloring Problem ($k$-BiCP)}}} \\
\textbf{Instance:}  & A hypergraph $H$ and an integer $k \geq 2$. \\
\textbf{Question:}  & Does $H$ admit a bipanchromatic $k$-coloring? \\
\end{tabular}
\end{table}

\begin{thm}
    \textsc{$k$-BiCP} in hypergraphs is NP-Complete.
\end{thm}
\begin{proof}
    First, \textsc{$k$-BiCP} is in NP. In fact, given a color assignment $\varphi$ on a hypergraph $H$, we can check in polynomial time that all $k$ colors occur within each hyperedge of $H$ and at least twice in $H$.We now demonstrate that it is NP-Complete through a polynomial reduction from the panchromatic $k$-coloring problem on hypergraphs. Given a hypergraph $H$ and an integer $k$, this problem asks whether there exists a coloring where each hyperedge of $H$ contains at least one vertex of each color. This problem is known to be NP-Complete \cite{garey} (recall that a panchromatic $2$-coloring is equivalent to a $2$-coloring).
    
    Given an instance of the panchromatic $k$-coloring problem on a hypergraph $H$, we construct an instance of \textsc{$k$-BiCP} on a hypergraph $H'$ as follows. $H'$ consists of two disjoint copies of $H$ denoted $H_1$ and $H_2$, and one hyperedge containing all the vertices of the two copies (see Figure~\ref{alpha7}). If $H$ is panchromatically $k$-colorable, then it is also the case for $H_1$ and $H_2$. Since these two copies are disjoint, from panchromatic $k$-colorings of $H_1$ and $H_2$ defined on the same set of $k$ colors, we obtain a bipanchromatic $k$-coloring of $H'$ (each of the $k$ colors appears at least once in each copy, thus at least twice in $H'$). Conversely, if $H'$ is bipanchromatically $k$-colorable, then by definition each hyperedge contains all $k$ colors. Since $H_1$ and $H_2$ are disjoint, this means that all of their hyperedges contain all $k$ colors. It implies that $H_1$ and $H_2$ are panchromatically $k$-colorable; otherwise, at least one hyperedge in $H'$ would not contain all colors, a contradiction. Since $H_1$ and $H_2$ are two copies of $H$, it suffices to color the vertices of $H$ with the coloring of one of the two copies. Hence, this constitutes a polynomial reduction from the panchromatic $k$-coloring existence problem to \textsc{$k$-BiCP}.
\end{proof}

\begin{figure}[ht]
    \begin{center}
    \begin{tikzpicture}[scale=0.95]
        \node[shape=circle,draw=black,scale=1.5,dashed] (1) at (0,0) {$H$};
        \node[shape=circle,draw=black,scale=1.5,dashed] (2) at (4,0) {$H_1$};
        \node[shape=circle,draw=black,scale=1.5,dashed] (3) at (7,0) {$H_2$};           
        \draw (5.5,0) ellipse (2.5cm and 1.5cm);
        \node[scale=1.5] at (5.5,-2) {$H'$};
        \draw[dashed] (5.5,-0.3) ellipse (3.5cm and 2.2cm);
    \end{tikzpicture}
    \caption{Hypergraphs $H$ and $H'$}
    \label{alpha7}
    \end{center}
\end{figure}
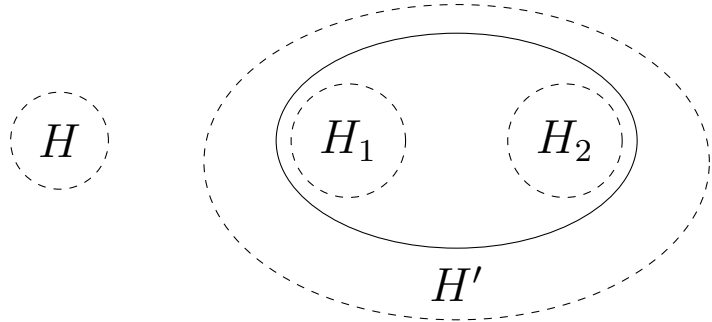

Now, we consider the $k$ CIST problem in split graphs; its recognition version is given by:
\begin{table}[H]
\raggedright
\begin{tabular}{p{2cm} p{10cm}}
\multicolumn{2}{l}{\textsc{\textbf{$k$ Completely Independent Spanning Trees Problem ($k$-CIST)}}} \\
\textbf{Instance:}  & A graph $G$ and an integer $k \geq 2$. \\
\textbf{Question:}  & Does $G$ have $k$ completely independent spanning trees? \\
\end{tabular}
\end{table}

\begin{thm}
    \textsc{$k$-CIST} in split graphs is NP-Complete.
\end{thm}
\begin{proof}
    First, \textsc{$k$-CIST} $\in$ NP. Indeed, given a split graph $G$ and a set of $k$ spanning trees $T_1,T_2,\dots,T_k$, we can check in polynomial time that all trees are edge-disjoint and for any vertex $x \in V(G)$, there is at most one spanning tree $T_i$ such that $d_{T_i}(x)>1$ (Theorem \ref{thm00}). Now, we show that \textsc{$k$-CIST} is NP-Complete in split graphs through a polynomial reduction from the panchromatic $k$-coloring problem in hypergraphs. 

    Given an instance of the panchromatic $k$-coloring problem on a hypergraph $H$, we construct an instance of \textsc{$k$-CIST} on a split graph $G'$ as follows. We first recall the construction of a split graph associated with a hypergraph. Let 
    $H = (D,\mathcal{E})$ be a hypergraph, the corresponding split graph $G(H) = (D \cup I,E)$ is defined by: 
    \begin{itemize}
        \item $I = \{ i_e | e \in \mathcal{E} \}$.
        \item $E(G(H)) = \{ (u,v) | u,v \in D, u \ne v \} \cup \{ (v,i_e) | e\in \mathcal{E}, v \in e \}$.
    \end{itemize}

    Thus, the vertex set $D$ induces a clique, and $I$ an independent set.
    
    To construct the graph $G'=(D'\cup I',E')$, we take two copies of $G(H)$, denoted $G_1 = (D_1 \cup I_1, E_1)$ and $G_2 = (D_2 \cup I_2, E_2)$. Then, we define:
    \begin{itemize}
        \item $D' = D_1 \cup D_2$ and $I' = I_1 \cup I_2$;
        \item $E'=E_1 \cup E_2 \cup \{ (x_1,x_2) | x_1 \in D_1, x_2 \in D_2\}$ ($D'$ is a complete subgraph).
    \end{itemize}
    
      The construction of $G'$ can be done in polynomial time, with a complexity of $O(n^2)$ (the clique on $D'$ contains $O(n^2)$ edges). See Figure~\ref{alpha10} for an illustration of the construction. 
    \begin{figure}[ht]
        \begin{center}
        \begin{tikzpicture}
            \node[shape=circle,draw=black] (h1) at (-1.5,-1) {};
            \node[shape=circle,draw=black] (h2) at (0,-1) {};
            \node[shape=circle,draw=black] (h3) at (0,0.5) {};
            \node[shape=circle,draw=black] (h4) at (-1.5,0.5) {};
            \node[scale=1.5] at (-0.8,-3) {$H$};
            \node[scale=1.5] at (4,-3) {$G(H)$};
            \node[scale=1.5] at (8.5,-3) {$G'$};
            \node[shape=circle,draw=black] (1) at (2.5,-1) {};
            \node[shape=circle,draw=black] (2) at (4,-1) {};
            \node[shape=circle,draw=black] (3) at (4,0.5) {};
            \node[shape=circle,draw=black] (4) at (2.5,0.5) {};
            \node[shape=circle,draw=black,scale=0.7] (8) at (5,1.5) {};
            \node[shape=circle,draw=black,scale=0.7] (9) at (5.5,-1) {};
            \begin{scope}[fill opacity=0]
            \filldraw[fill=yellow!70] ($(h3)+(0.4,0.2)$)
                to[out=-90,in=90] ($(h2) + (0.4,-0.2)$)
                to[out=-90,in=-90] ($(h2) + (-0.4,-0.2)$) 
                to[out=90,in=-90] ($(h3) + (-0.4,0.2)$)  
                to[out=90,in=90] ($(h3) + (0.4,0.2)$)   ;
            \end{scope}    
            \begin{scope}[fill opacity=0]
            \filldraw[fill=yellow!70] ($(h3)+(0,0.5)$)
                to[out=180,in=0] ($(h4) + (0,0.5)$)    
                to[out=180,in=90] ($(h4) + (-0.5,0)$)
                to[out=-90,in=90] ($(h2) + (-0.6,0)$) 
                to[out=-90,in=180] ($(h2) + (0,-0.5)$) 
                to[out=0,in=-90] ($(h2) + (0.6,0)$)
                to[out=90,in=-90] ($(h3) + (0.6,0)$)        
                to[out=90,in=0] ($(h3) + (0,0.5)$)   ;
            \end{scope} 
            \begin{scope}[fill opacity=0]
            \filldraw[fill=yellow!70] ($(h3)+(0,0.6)$)
                to[out=180,in=0] ($(h1) + (0,0.5)$)    
                to[out=180,in=90] ($(h1) + (-0.5,0)$)
                to[out=-90,in=180] ($(h1) + (0,-0.6)$)
                to[out=0,in=180] ($(h2) + (0,-0.6)$) 
                to[out=0,in=-90] ($(h2) + (0.7,0)$)
                to[out=90,in=-90] ($(h3) + (0.7,0)$)        
                to[out=90,in=0] ($(h3) + (0,0.6)$)   ;
            \end{scope}    
            \path [-] (1) edge node {} (3);    
            \path [-] (1) edge node {} (2);    
            \path [-] (1) edge node {} (4);     
            \path [-] (2) edge node {} (3);    
            \path [-] (2) edge node {} (4);
            \path [-] (3) edge node {} (4);
            \path [-] (8) edge node {} (3);
            \path [-] (8) edge node {} (4);
            \path [-] (8) edge node {} (2);
            \path [-] (9) edge node {} (2);
            \path [-] (9) edge node {} (3);
            \node[shape=circle,draw=black,fill=white] (01) at (7,-2) {};
            \node[shape=circle,draw=black,fill=white] (02) at (8.5,-2) {};
            \node[shape=circle,draw=black,fill=white,] (03) at (8.5,-0.5) {};
            \node[shape=circle,draw=black,fill=white] (04) at (7,-0.5) {};
            \node[shape=circle,draw=black,fill=white] (11) at (7,1.5) {};
            \node[shape=circle,draw=black,fill=white] (12) at (8.5,1.5) {};
            \node[shape=circle,draw=black,fill=white] (13) at (8.5,3) {};
            \node[shape=circle,draw=black,fill=white] (14) at (7,3) {};
            \node[shape=circle,draw=black,scale=0.7] (09) at (10,-2) {};
            \node[shape=circle,draw=black,scale=0.7] (19) at (10,1.5) {};
            \node[shape=circle,draw=black,scale=0.7] (08) at (9.5,0.5) {};
            \node[shape=circle,draw=black,scale=0.7] (18) at (9.5,4) {};
            \path [-] (01) edge node {} (02);    
            \path [-] (01) edge node {} (03);    
            \path [-] (01) edge node {} (04);  
            \path [-] (01) edge node {} (12);    
            \path [-] (01) edge node {} (13);     
            \path [-] (02) edge node {} (03);    
            \path [-] (02) edge node {} (04);
            \path [-] (02) edge node {} (11);            
            \path [-] (02) edge node {} (14);    
            \path [-] (03) edge node {} (04);
            \path [-] (03) edge node {} (11);
            \path [-] (03) edge node {} (12);
            \path [-] (03) edge node {} (14);
            \path [-] (04) edge node {} (11);
            \path [-] (04) edge node {} (12);
            \path [-] (04) edge node {} (13);
            \path [-] (11) edge node {} (13);    
            \path [-] (11) edge node {} (12);    
            \path [-] (11) edge node {} (14);     
            \path [-] (12) edge node {} (13);    
            \path [-] (12) edge node {} (14);
            \path [-] (13) edge node {} (14);
            \draw[-] (01) to[out=70, in= -70] (11); 
            \draw[-] (01) to[out=70, in= -70] (14); 
            \draw[-] (04) to[out=70, in= -70] (14); 
            \draw[-] (02) to[out=110, in= -110] (12); 
            \draw[-] (02) to[out=110, in= -110] (13); 
            \draw[-] (03) to[out=110, in= -110] (13); 
            \path [-] (02) edge node {} (09);    
            \path [-] (03) edge node {} (09);     
            \path [-] (12) edge node {} (19);    
            \path [-] (13) edge node {} (19);     
            \path [-] (02) edge node {} (08);    
            \path [-] (03) edge node {} (08);  
            \path [-] (04) edge node {} (08);  
            \path [-] (12) edge node {} (18);    
            \path [-] (13) edge node {} (18);    
            \path [-] (14) edge node {} (18);  
            \draw[-] (9) to[out=-150, in= -30] (1); 
            \draw[-] (09) to[out=-150, in= -30] (01); 
            \draw[-] (19) to[out=-150, in= -30] (11); 
        \end{tikzpicture}
        \caption{A polynomial-time reduction from the panchromatic $2$-coloring problem in hypergraph $H$ to \textsc{$2$-CIST} in split graph $G'$}
        \label{alpha10}
        \end{center}
    \end{figure}
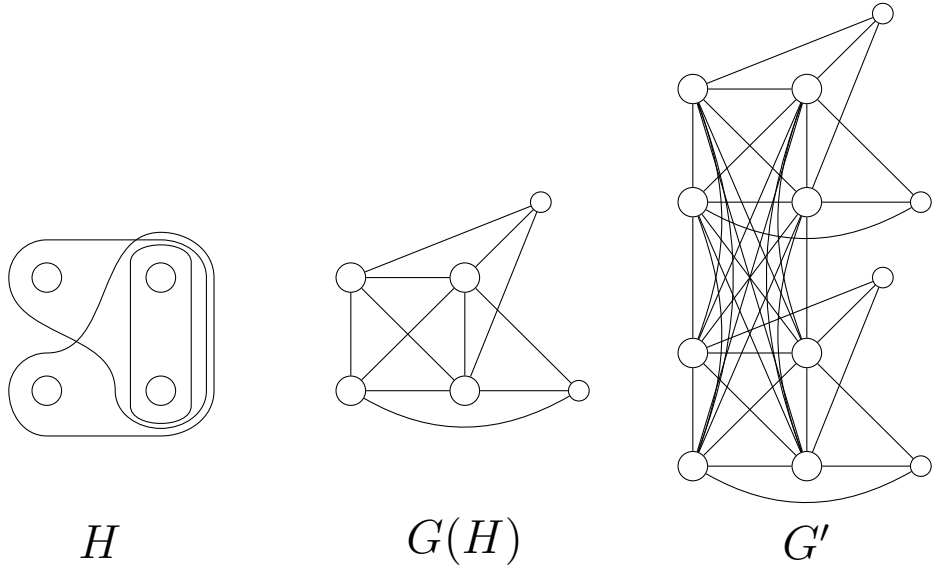    

    $\Longrightarrow$ We suppose that $\varphi$ is a panchromatic $k$-coloring of $H$. We use the same coloring $\varphi$ for $H(G_1)$ and $H(G_2)$, as they are isomorphic to $H$. Since each vertex of $H(G_1)$ has the same color as its corresponding vertex in $H(G_2)$, each color appears at least twice in $H(G')$. Therefore, $H(G')$ is bipanchromatically $k$-colorable. According to Theorem \ref{lemme1}, $G'$ has $k$ completely independent spanning trees. 
    
    $\Longleftarrow$ Conversely, assume that $G'$ admits $k$ completely independent spanning trees $T_1,T_2,\dots,T_k$. We construct a panchromatic $k$-coloring $\varphi'$ of the hypergraph $H(G')$ as follows: We assign color $i$ to each internal vertex of $T_i$ in $D'$, $ 1 \leq i \leq k$. As each vertex $y_1 \in E_1$ and each vertex $y_2 \in E_2$ is covered by internal vertices of all $k$ spanning trees in $G'$, each hyperedge $e_{y_1}$ and each $e_{y_2}$ contains all $k$ colors.
    Therefore $H(G_1)$ and $H(G_2)$ are both panchromatically $k$-colorable. Let $\varphi$ be a panchromatic $k$-coloring of $H(G_1)$. As $H(G)$ is isomorphic to $H(G_1)$, we use the same coloring $\varphi$ for $H(G)$. It follows that $H(G)$ is panchromatically $k$-colorable. Figure~\ref{alpha9} illustrates an example of the mapping from $H$ to $G'$. Color $1$ which is blue and color $2$ which is red form a panchromatic $2$-coloring in $H$. The blue tree $T_1$ and the red tree $T_2$ are $2$ CIST in $G'$. Note that the vertices $y_1$ and $y_2$ are covered by both trees $T_1$ and $T_2$ in $G'$, and the hyperedges $e_{y_1}$ and $e_{y_2}$ contains both colors.

    \begin{figure}[ht]
        \begin{center}
        \begin{tikzpicture}
            \node[shape=circle,draw=black,fill=red] (h1) at (-1.5,-1.5) {};
            \node[shape=circle,draw=black,fill=blue] (h2) at (0,-1.5) {};
            \node[shape=circle,draw=black,fill=red] (h3) at (0,0) {};
            \node[shape=circle,draw=black,fill=blue] (h4) at (-1.5,0) {};
            \node[scale=1.5] at (-0.8,-3) {$H$};
            \node[scale=1.5] at (4,-3) {$H(G')$};
            \node[scale=1.5] at (8.8,-3) {$G'$};
            \begin{scope}[fill opacity=0]
            \filldraw[fill=yellow!70] ($(h3)+(0.4,0.2)$)
                to[out=-90,in=90] ($(h2) + (0.4,-0.2)$)
                to[out=-90,in=-90] ($(h2) + (-0.4,-0.2)$) 
                to[out=90,in=-90] ($(h3) + (-0.4,0.2)$)  
                to[out=90,in=90] ($(h3) + (0.4,0.2)$)   ;
            \end{scope}    
            \begin{scope}[fill opacity=0]
            \filldraw[fill=yellow!70] ($(h3)+(0,0.5)$)
                to[out=180,in=0] ($(h4) + (0,0.5)$)    
                to[out=180,in=90] ($(h4) + (-0.5,0)$)
                to[out=-90,in=90] ($(h2) + (-0.6,0)$) 
                to[out=-90,in=180] ($(h2) + (0,-0.5)$) 
                to[out=0,in=-90] ($(h2) + (0.6,0)$)
                to[out=90,in=-90] ($(h3) + (0.6,0)$)        
                to[out=90,in=0] ($(h3) + (0,0.5)$)   ;
            \end{scope} 
            \begin{scope}[fill opacity=0]
            \filldraw[fill=yellow!70] ($(h3)+(0,0.6)$)
                to[out=180,in=0] ($(h1) + (0,0.5)$)    
                to[out=180,in=90] ($(h1) + (-0.5,0)$)
                to[out=-90,in=180] ($(h1) + (0,-0.6)$)
                to[out=0,in=180] ($(h2) + (0,-0.6)$) 
                to[out=0,in=-90] ($(h2) + (0.7,0)$)
                to[out=90,in=-90] ($(h3) + (0.7,0)$)        
                to[out=90,in=0] ($(h3) + (0,0.6)$)   ;
            \end{scope}    
            \node[] at (10.5,4.5) {$T_2$};
            \node[] at (10.5,5) {$T_1$};
            \path [-,thick,blue] (11,5) edge node {} (11.5,5);      
            \path [-,thick,red] (11,4.5) edge node {} (11.5,4.5);      
            \node[shape=circle,draw=black,fill=white,fill=red] (01) at (7.3,-2) {};
            \node[shape=circle,draw=black,fill=white,fill=blue] (02) at (8.8,-2) {};
            \node[shape=circle,draw=black,fill=white,fill=red] (03) at (8.8,-0.5) {};
            \node[shape=circle,draw=black,fill=white,fill=blue] (04) at (7.3,-0.5) {};
            \node[shape=circle,draw=black,fill=white,fill=red] (11) at (7.3,1.5) {};
            \node[shape=circle,draw=black,fill=white,fill=blue] (12) at (8.8,1.5) {};
            \node[shape=circle,draw=black,fill=white,fill=red] (13) at (8.8,3) {};
            \node[shape=circle,draw=black,fill=white,fill=blue] (14) at (7.3,3) {};
            \node[shape=circle,draw=black,scale=0.8] (08) at (9.5,0.5) {};
            \node[shape=circle,draw=black,scale=0.7] (09) at (10.3,-2) {$y_1$};
            \node[shape=circle,draw=black,scale=0.7] (18) at (9.5,4) {};
            \node[shape=circle,draw=black,scale=0.7] (19) at (10.3,1.5) {$y_2$};
            \path [-,red,thick] (01) edge node {} (02);    
            \path [-] (01) edge node {} (03);    
            \path [-,blue,thick] (01) edge node {} (04);  
            \path [-] (01) edge node {} (12);    
            \path [-] (01) edge node {} (13);     
            \path [-,blue,thick] (02) edge node {} (03);    
            \path [-,blue,thick] (02) edge node {} (04);
            \path [-] (02) edge node {} (11);            
            \path [-] (02) edge node {} (14);    
            \path [-,red,thick] (03) edge node {} (04);
            \path [-] (03) edge node {} (11);
            \path [-] (03) edge node {} (12);
            \path [-] (03) edge node {} (14);
            \path [-] (04) edge node {} (11);
            \path [-] (04) edge node {} (12);
            \path [-] (04) edge node {} (13);
            \path [-,red,thick] (11) edge node {} (13);    
            \path [-,red,thick] (11) edge node {} (12);    
            \path [-,blue,thick] (11) edge node {} (14);     
            \path [-,blue,thick] (12) edge node {} (13);    
            \path [-] (12) edge node {} (14);
            \path [-,red,thick] (13) edge node {} (14);
            \draw[-,red,thick] (01) to[out=70, in= -70] (11); 
            \draw[-] (01) to[out=70, in= -70] (14); 
            \draw[-,blue,thick] (04) to[out=70, in= -70] (14); 
            \draw[-,blue,thick] (02) to[out=110, in= -110] (12); 
            \draw[-] (02) to[out=110, in= -110] (13); 
            \draw[-,red,thick] (03) to[out=110, in= -110] (13); 
            \path [-,blue,thick] (02) edge node {} (09);    
            \path [-,red,thick] (03) edge node {} (09);     
            \path [-,blue,thick] (12) edge node {} (19);    
            \path [-,red,thick] (13) edge node {} (19);     
            \path [-,blue,thick] (02) edge node {} (08);    
            \path [-,red,thick] (03) edge node {} (08);  
            \path [-] (04) edge node {} (08);  
            \path [-,blue,thick] (12) edge node {} (18);    
            \path [-,red,thick] (13) edge node {} (18);    
            \path [-] (14) edge node {} (18);  
            \draw[-] (09) to[out=-150, in= -30] (01); 
            \draw[-] (19) to[out=-150, in= -30] (11); 
            \node[shape=circle,draw=black,fill=red] (21) at (3.3,-1.5) {};
            \node[shape=circle,draw=black,fill=blue] (22) at (4.8,-1.5) {};
            \node[shape=circle,draw=black,fill=red] (23) at (4.8,0) {};
            \node[shape=circle,draw=black,fill=blue] (24) at (3.3,0) {};
            \begin{scope}[fill opacity=0]
            \filldraw[fill=yellow!70] ($(23)+(0.4,0.2)$)
                to[out=-90,in=90] ($(22) + (0.4,-0.2)$)
                to[out=-90,in=-90] ($(22) + (-0.4,-0.2)$) 
                to[out=90,in=-90] ($(23) + (-0.4,0.2)$)  
                to[out=90,in=90] ($(23) + (0.4,0.2)$)   ;
            \end{scope}    
            \begin{scope}[fill opacity=0]
            \filldraw[fill=yellow!70] ($(23)+(0,0.5)$)
                to[out=180,in=0] ($(24) + (0,0.5)$)    
                to[out=180,in=90] ($(24) + (-0.5,0)$)
                to[out=-90,in=90] ($(22) + (-0.6,0)$) 
                to[out=-90,in=180] ($(22) + (0,-0.5)$) 
                to[out=0,in=-90] ($(22) + (0.6,0)$)
                to[out=90,in=-90] ($(23) + (0.6,0)$)        
                to[out=90,in=0] ($(23) + (0,0.5)$)   ;
            \end{scope} 
            \begin{scope}[fill opacity=0]
            \filldraw[fill=yellow!70] ($(23)+(0,0.6)$)
                to[out=180,in=0] ($(21) + (0,0.5)$)    
                to[out=180,in=90] ($(21) + (-0.5,0)$)
                to[out=-90,in=180] ($(21) + (0,-0.6)$)
                to[out=0,in=180] ($(22) + (0,-0.6)$) 
                to[out=0,in=-90] ($(22) + (0.7,0)$)
                to[out=90,in=-90] ($(23) + (0.7,0)$)        
                to[out=90,in=0] ($(23) + (0,0.6)$)   ;
            \end{scope}        
            \node[shape=circle,draw=black,fill=red] (31) at (3.3,1.5) {};
            \node[shape=circle,draw=black,fill=blue] (32) at (4.8,1.5) {};
            \node[shape=circle,draw=black,fill=red] (33) at (4.8,3) {};
            \node[shape=circle,draw=black,fill=blue] (34) at (3.3,3) {};
            \begin{scope}[fill opacity=0]
            \filldraw[fill=yellow!70] ($(33)+(0.4,0.2)$)
                to[out=-90,in=90] ($(32) + (0.4,-0.2)$)
                to[out=-90,in=-90] ($(32) + (-0.4,-0.2)$) 
                to[out=90,in=-90] ($(33) + (-0.4,0.2)$)  
                to[out=90,in=90] ($(33) + (0.4,0.2)$)   ;
            \end{scope}    
            \begin{scope}[fill opacity=0]
            \filldraw[fill=yellow!70] ($(33)+(0,0.5)$)
                to[out=180,in=0] ($(34) + (0,0.5)$)    
                to[out=180,in=90] ($(34) + (-0.5,0)$)
                to[out=-90,in=90] ($(32) + (-0.6,0)$) 
                to[out=-90,in=180] ($(32) + (0,-0.5)$) 
                to[out=0,in=-90] ($(32) + (0.6,0)$)
                to[out=90,in=-90] ($(33) + (0.6,0)$)        
                to[out=90,in=0] ($(33) + (0,0.5)$)   ;
            \end{scope} 
            \begin{scope}[fill opacity=0]
            \filldraw[fill=yellow!70] ($(33)+(0,0.6)$)
                to[out=180,in=0] ($(31) + (0,0.5)$)    
                to[out=180,in=90] ($(31) + (-0.5,0)$)
                to[out=-90,in=180] ($(31) + (0,-0.6)$)
                to[out=0,in=180] ($(32) + (0,-0.6)$) 
                to[out=0,in=-90] ($(32) + (0.7,0)$)
                to[out=90,in=-90] ($(33) + (0.7,0)$)        
                to[out=90,in=0] ($(33) + (0,0.6)$)   ;
            \end{scope}       
        \begin{scope}[fill opacity=0,dashed]
        \filldraw[fill=yellow!70] ($(21)+(-0.9,-0.1)$)
            to[out=-90,in=-90] ($(22) + (0.9,-0.1)$)    
            to[out=90,in=-90] ($(33) + (0.9,0.2)$)
            to[out=90,in=90] ($(34) + (-0.8,0.2)$)
            to[out=-90,in=90] ($(21) + (-0.8,-0.1)$) ;
        \end{scope}    
        \end{tikzpicture}
        \caption{A mapping of a panchromatic $2$-coloring solution in $H$ to a $2$-CIST solution in $G'$}
        \label{alpha9}
        \end{center}
    \end{figure}
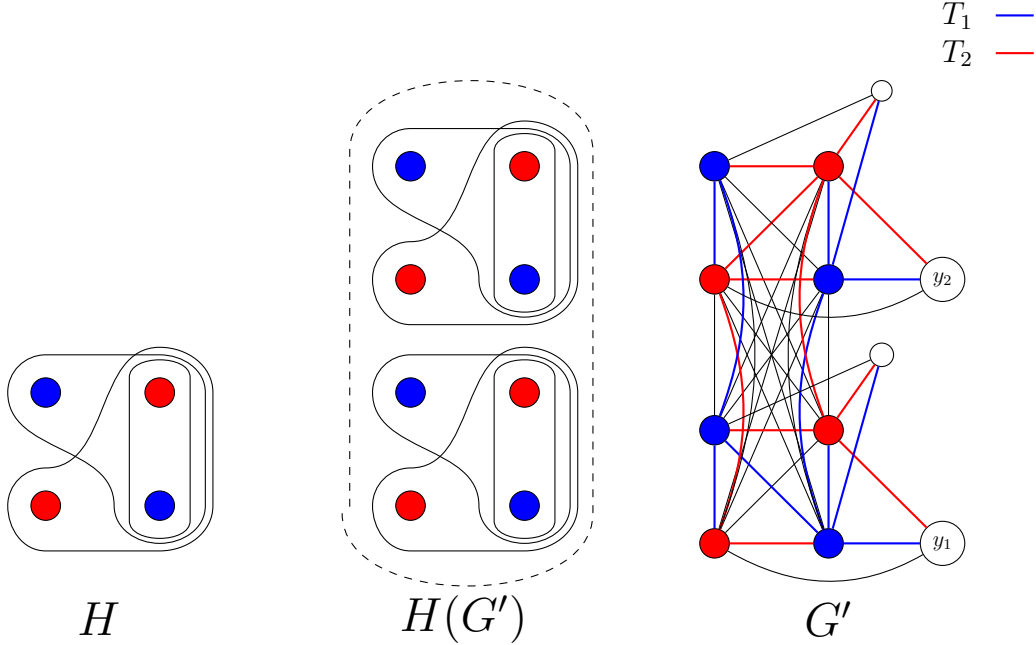    
    Hence, this constitutes a polynomial reduction from the panchromatic $k$-coloring problem to \textsc{$k$-CIST} in split graphs.
\end{proof}

\section{Concluding remarks}\label{sec5}
We have shown that the number of CIST in a split graph $G$ is at most $\chi_p^2$ or $\chi_p^2 + 1$, where $\chi_p^2=\chi_p^2(H(G))$ is the bipanchromatic number of $H(G)$. Moreover, from the proof of Theorem \ref{thm4}, we have $ \chi_p(H) - \left\lceil \frac{\alpha}{2} \right\rceil  \le \chi^2_p(H) \le \chi_p(H)$ for every hypergraph $H$, where $\alpha = \alpha_{\chi_p(H)}(H)$ is the minimum number of unique colors among all panchromatic $\chi_p(H)$-colorings of $H$. As bipanchromatic coloring is a new type of coloring, it would be useful to prove a stronger result relating the panchromatic and bipanchromatic numbers of a hypergraph $H$. Specifically, it would be nice to prove that, given any panchromatic $\chi_p(H)$-coloring of $H$ minimizing $\alpha$, pairing the unique colors two by two until no unique color remains yields exactly a bipanchromatic $\chi^2_p$-coloring, giving for any hypergraph $H$:
\begin{equation}
    \label{eq3}
    \chi^2_p(H) = \chi_p(H) -  \left\lceil \frac{\alpha_{\chi_p(H)}(H)}{2} \right\rceil
\end{equation}
To test Equation~(\ref{eq3}), we develop integer linear programming formulations (see Appendix \ref{appA}) for the panchromatic and bipanchromatic coloring problems and the minimization of unique colors problem. We randomly generate a set of $100$ hypergraphs with $n$ vertices and $m$ hyperedges for each pair $(n,m)$ such that $n \in \{ 4,5,\dots,13 \}$ and $m \in \{ n-1,n, n+1, \dots,n+8 \}$. We observed that each generated hypergraph satisfies (\ref{eq3}). 

These positive tests, along with the natural process of obtaining a bipanchromatic coloring from an optimal panchromatic coloring by pairing up the unique colors, lead us to propose the following conjecture:
\begin{conjecture}
    Every hypergraph $H$ satisfies $\chi^2_p (H)= \chi_p (H) -  \left\lceil \frac{\alpha_{\chi_p(H)}(H)}{2} \right\rceil$.
    \label{conj}
\end{conjecture}
This conjecture, if true, would yield tight bounds on the maximum number of CIST in a split graph. By Theorems \ref{lemme1}, \ref{thm4} and \ref{thm7}, we have:
\begin{prop}
    Let $G = (D \cup I,E)$ be a split graph. If Conjecture \ref{conj} is true, then the maximum number $M$ of CIST of $G$ satisfies $$\chi_p (H(G)) -  \left\lceil \frac{\alpha}{2} \right \rceil \leq M \leq \chi_p (H(G)) -  \left\lceil \frac{\alpha}{2}\right \rceil + 1,$$ 
    where $\alpha=\alpha_{\chi_p(H(G))}(H(G))$.
    \label{prop1}
\end{prop}

\bibliographystyle{abbrvnat}
\bibliography{reference.bib}

\appendix
\appendixpage
\addappheadtotoc
\section{Linear modeling} \label{appA}

Let $H = (V,\mathcal{E})$ be a hypergraph such that $|D| = n$ and $|I| = m$. Let $C = \{1,2,\dots,c\}$ denote a set of colors. Let $A = (a_{ij})$ be the incidence matrix of $H$ such that $a_{ij} = 1$ if the hyperedge $i$ contains the vertex $j$, otherwise $a_{ij} = 0$. 
\subsection{The panchromatic coloring problem}
The goal is to determine the panchromatic number of $H$. In order to do that we introduce the variable $k_p$ that indicates if a color $p$ is used at least once, as well as $x_{jp}$ that indicates if a vertex $j$ is colored by the color $p$. These variables are defined as:
\begin{center}
$ k_p \ \ = \left\lbrace\begin{tabular}{p{1cm} p{10cm}}
1 &  \mbox{if the color $p$ is used by at least one vertex, $ p \in [1,c], $}\\
0 & \mbox{otherwise}. \\
\end{tabular}\right.
$
\end{center}
\begin{center}
$ x_{jp} \ \ = \left\lbrace\begin{tabular}{p{1cm} p{10cm}}
1 &  \mbox{if the vertex $j$ is assigned the color $p$, $j \in [1,n],$ $p \in [1,c], $}\\
0 & \mbox{otherwise}. \\
\end{tabular}\right.
$
\end{center}
Since the value of $c$ cannot be greater than $n$, let us take $c =n$. Our ILP formulation of the panchromatic coloring problem is: 

Maximize $\sum\limits_{p=1}^n k_{p}$

subject to 
\begin{align}[left=\empheqlbrace]
& x_{jp} \le k_p && \forall j \in [1,n],\ p \in [1,c] \label{a} \\
& \sum\limits_{p = 1}^c x_{jp}=1 && \forall j \in [1,n] \label{b}\\
& \sum\limits_{j = 1}^n x_{jp} a_{ij}\ge k_{p} && \forall i \in [1,m], p \in [1,c] \label{c}\\
& x_{jp},k_{p}, a_{ij} \in \{0,1\} && \forall i \in [1,m], j \in [1,n], p \in [1,c] 
\end{align}
Eq. (\ref{a}) guarantee that if a color is not used, then it cannot be assigned to any vertex. Eq. (\ref{b}) ensures that only one color is assigned to each vertex. Also, by Eq. (\ref{c}), if a color is used, it must appear in every hyperedge.

\subsection{The bipanchromatic coloring problem}
To determine the bipanchromatic number of $H$, the previous model can be modified by adding the constraint that each color appears at least twice in the graph: 
        \begin{center}
        \begin{tabular}{p{3.1cm} p{3.5cm}}
        $ \sum\limits_{j = 1}^n x_{jp}  \geq 2 \times k_{p} $ & $\forall p \in [1,\chi_p] $   \\
        \end{tabular}
        \end{center} 
Since $\chi^2_p \leq \chi_p$, let us set $c = \chi_p$. This implies that in order to determine the bipanchromatic number of $H$, the panchromatic number of $H$ has to be found first. The objective function is the same as for the panchromatic coloring problem.

\subsection{The minimum number of unique colors problem}
To determine the minimum number of unique colors $\alpha_{\chi_p(H)}(H)$ among all panchromatic $\chi_p$-colorings of the hypergraph $H$, we introduce the variable $v_p$ that indicates if a color $p$ is unique, it is defined as:
\begin{center}
$ v_{p} \ \ = 
\left\lbrace\begin{tabular}{p{1cm} p{10cm}}
1 & \mbox{if the color $p$ is a unique color, $p \in [1,\chi_p], $}\\
0 & \mbox{otherwise}. \\
\end{tabular}\right.
$\end{center}
Our ILP formulation of the minimum number of unique colors problem is: 

Minimize $\sum\limits_{p=1}^{\chi_p} v_{p}$

subject to 
\begin{align}[left=\empheqlbrace]
     & \sum\limits_{j = 1}^n x_{jp} a_{ij} \geq 1 && \forall i \in [1,m], p \in [1,\chi_p] \label{a1}\\
     & \sum\limits_{p = 1}^{\chi_p} x_{jp} = 1 && \forall j \in [1,n] \label{b1}\\   
     & \sum\limits_{j = 1}^n x_{jp} + v_p \geq 2 && \forall p \in [1,\chi_p] \label{c1}\\
     & x_{jp},a_{ij},v_j \in \{0,1\} && \forall i \in [1,m], j \in [1,n], p \in [1,\chi_p] 
\end{align}

Eq. (\ref{c1}) states that if a color $p$ is unique, \textit{i.e.}, $\sum_{j = 1}^n x_{jp} = 1$, then it implies that $v_p = 1$. However, if the color $p$ is not unique, that is, $\sum_{j = 1}^n x_{jp} \ge 2$, then, since the objective of the problem is to minimize $\sum v_p$, it follows that $v_p = 0$.
\end{document}